\documentclass[11pt,reqno]{amsart}

\usepackage{hyperref}
\usepackage{graphicx}
\usepackage{color}
\usepackage{amsfonts,amssymb}
\usepackage{bbm} 

\usepackage{mathrsfs}
\usepackage{a4wide}
\usepackage{psfrag}
\usepackage{pdfsync}

\newtheorem{neu}{}[section]

\newtheorem*{Cor*}{Corollary}

\newtheorem{Thm}[neu]{Theorem}
\newtheorem{Theorem}{Theorem}
\newtheorem*{Thm*}{Theorem}
\newtheorem{Prop}[neu]{Proposition}
\newtheorem*{Prop*}{Proposition}
\theoremstyle{definition}
\newtheorem{Lemma}[neu]{Lemma}
\newtheorem*{Rmk*}{Remark}
\newtheorem{Rmk}[neu]{Remark}

\newtheorem*{Ex*}{Example}

\newtheorem*{Qu*}{Question}

\newtheorem*{Claim*}{Claim}
\newtheorem{Def}[neu]{Definition}

\newtheorem*{Conv*}{Convention}

\newcommand{\N}{\mathbb{N}}

\newcommand{\Z}{\mathbb{Z}}
\newcommand{\R}{\mathbb{R}}
\newcommand{\C}{\mathbb{C}}

\newcommand{\pf}{\longrightarrow}

\newcommand{\hpf}{\hookrightarrow}

\newcommand{\Morse}{\mu_{\mathrm{Morse}}}

\newcommand{\ind}{\mathrm{ind\,}}

\newcommand{\om}{\omega}

\newcommand{\ev}{\mathrm{ev}}
\newcommand{\evt}{\widetilde{\mathrm{ev}}}
\newcommand{\EVt}{\widetilde{\mathrm{EV}}}

\newcommand{\A}{\mathcal{A}}

\newcommand{\F}{\mathcal{F}}

\newcommand{\M}{\mathcal{M}}
\newcommand{\Mt}{\widetilde{\mathcal{M}}}

\renewcommand{\L}{\mathscr{L}}

\renewcommand{\H}{\mathrm{H}}

\newcommand{\Ham}{\mathrm{Ham}}

\newcommand{\CM}{\mathrm{CM}}

\newcommand{\Crit}{\mathrm{Crit}}

\newcommand{\p}{\partial}
\newcommand{\cl}{\mathrm{cl}}
\newcommand{\LI}{\mathcal{LI}}

\newcommand{\beq}{\begin{equation}}
\newcommand{\beqn}{\begin{equation}\nonumber}
\newcommand{\eeq}{\end{equation}}

\newcommand{\bea}{\begin{equation}\begin{aligned}}
\newcommand{\bean}{\begin{equation}\begin{aligned}\nonumber}
\newcommand{\eea}{\end{aligned}\end{equation}}

\newcommand{\Mp}{\mathfrak{M}}

\numberwithin{equation}{section}

\definecolor{Urs}{rgb}{0,.7,0}
\definecolor{Peter}{rgb}{0,0,1}
\definecolor{red}{rgb}{1,0,0}

\begin{document}
\title{Cup-length estimates for leaf-wise intersections}
\author{Peter Albers}
\author{Al Momin}
\address{
    Peter Albers\\
    Department of Mathematics\\
    Purdue University}
\email{palbers@math.purdue.edu}
\address{
    Al Momin\\
    Department of Mathematics\\
    Purdue University}
\email{amomin@math.purdue.edu}
\begin{abstract}
We prove that on a restricted contact type hypersurface the number of leaf-wise intersections is bounded from below by a certain cup-length.
\end{abstract}
\maketitle

\section{Introduction}

Let $(M,\om)$ be a  symplectic manifold and $\Sigma\subset M$ be a hypersurface. Then $\Sigma$ is foliated by the characteristic foliation induced by the line bundle $\ker\om|_\Sigma\to\Sigma$. We denote by $L_x$ the leaf through $x\in\Sigma$. Let $\psi\in\Ham(M)$ be a Hamiltonian diffeomorphisms. Then a leaf-wise intersection is a point $x\in\Sigma$ with the property that $\psi(x)\in L_x$. 

\begin{Def}\label{def:basic_quantities}
Let $\iota:\Sigma\hpf M$ be the inclusion map.
\begin{enumerate}
\item We denote by
\beq
\cl(\Sigma, M):=\max\{k\mid\exists\; a_1,...,a_k\in\H^{\geq1}(M;\Z/2)\text{ with }\iota^*\big(a_1\cup\ldots\cup a_k\big)\neq0\}
\eeq
the relative cup-length of $\Sigma$ in $M$.
\item Suppose $\omega = d\lambda$ and that $\Sigma$ is of restricted contact type, i.e.~$\alpha := \iota^*\lambda$ is a contact form on $\Sigma$.  Then we denote by $\wp(\Sigma,\alpha)>0$ the minimal period of a Reeb orbit of $(\Sigma,\alpha)$ which is contractible in $M$. If there exists no such Reeb orbit we set $\wp(\Sigma,\alpha)=\infty$. 
\item We denote by $\Ham_c(M)$ the space of Hamiltonian diffeomorphisms generated by compactly supported, time dependent Hamiltonian functions and by $||\psi||$ the Hofer norm.
\item $(M,\om)$ is called convex at infinity if it is isomorphic to the symplectization of a compact contact manifold at infinity.
\end{enumerate}
\end{Def}

\begin{Theorem}\label{thm:main}
Let $\Sigma\subset (M,d\lambda)$ be a closed, bounding, restricted contact type hypersurface and $(M,d\lambda)$ be convex at infinity. If $\psi\in\Ham_c(M)$ satisfies
$||\psi||<\wp(\Sigma,\alpha)$
then
\beq
\nu_{\mathrm{leaf}}(\psi):=\#\{\text{leaf-wise intersections of }\psi\}\geq\cl(\Sigma,M)+1\;.
\eeq
\end{Theorem}

\begin{Rmk}
The search for leaf-wise intersections was initiated by Moser in \cite{Moser_A_fixed_point_theorem_in_symplectic_geometry} and pursued further in 
\cite{Banyaga_On_fixed_points_of_symplectic_maps,
Hofer_On_the_topological_properties_of_symplectic_maps,
Ekeland_Hofer_Two_symplectic_fixed_point_theorems_with_applications_to_Hamiltonian_dynamics,
Ginzburg_Coisotropic_intersections,
Dragnev_Symplectic_rigidity_symplectic_fixed_points_and_global_perturbations_of_Hamiltonian_systems,
Albers_Frauenfelder_Leafwise_intersections_and_RFH,
Ziltener_coisotropic, Albers_Frauenfelder_Leafwise_Intersections_Are_Generically_Morse,
Albers_McLean_SH_and_infinitley_many_LI,
Gurel_leafwise_coisotropic_intersection,Kang_Existence_of_leafwise_intersection_points_in_the_unrestricted_case, Albers_Frauenfelder_Remark_on_a_Thm_by_Ekeland_Hofer,
Albers_Frauenfelder_Spectral_invariants_in_RFH, 
Merry_On_the_RFH_of_twisted_cotangent_bundles}.  We refer to \cite{Albers_Frauenfelder_Leafwise_Intersections_Are_Generically_Morse} for a brief history of the problem.
\end{Rmk}

\begin{Rmk}
The example $S^1\subset\C$ shows that if $||\psi||>\wp(\Sigma,\alpha)$ then $\nu_{\mathrm{leaf}}(\psi)=0$ is possible.
\end{Rmk}

\begin{Rmk}
The relative cup-length is smaller than the cup-length of $M$ and $\Sigma$. In the case $\Sigma=S^*B\subset T^*B$ is a unit cotangent bundle we have $\cl(S^*B, T^*B)\geq\cl(B)-1$ by examining the Gysin sequence.
\end{Rmk}

\begin{Rmk}
For simplicity we shall use $\Z/2$-coefficients for all (co-)homology theories in this paper. We expect that Theorem \ref{thm:main} continues to hold with $\Z$-coefficients.  
\end{Rmk}

\begin{Rmk}
Cup-length estimates have been established for Lagrangian intersections by Floer \cite{Floer_cuplength_estimates_on_Lagrangian_intersections}, Hofer \cite{Hofer_Lusterik_Schnirelman_theroy_for_Lagrangian_intersections}, and Liu \cite{Liu_cuplength_estimate_for_Lagrangian_intersections} in terms of the cup-length of the Lagrangian submanifold and for fixed points of Hamiltonian diffeomorphisms by Schwarz \cite{Schwarz_A_quantum_cup_length_estimate_for_symplectic_fixed_points} in terms of the quantum cup-length of the symplectic manifold.
\end{Rmk}

\section{Leaf-wise intersections  and the Rabinowitz action functional}

Let $(M,\om)$ be a symplectic manifold and $f\in C^\infty(M)$ an autonomous Hamiltonian function. Since energy is preserved the hypersurface $\Sigma:=f^{-1}(0)$ is invariant under the Hamiltonian flow $\phi_f^t$ of $f$. The Hamiltonian flow $\phi_f^t$ is generated by the Hamiltonian vector field $X_f$ which is uniquely defined by the equation $\om(X_f,\cdot)=df$. If $0$ is a regular value of $f$ the hypersurface is a coisotropic submanifold which is foliated by 1-dimensional isotropic leaves, see \cite[Section 3.3]{McDuff_Salamon_introduction_symplectic_topology}. If we denote by $L_x$ the leaf through $x\in\Sigma$ we have the equality
\beq
L_x=\bigcup_{t\in\R}\phi_f^t(x)\;.
\eeq
Given a time-dependent Hamiltonian function $H:[0,1]\times M\pf\R$ with Hamiltonian flow $\phi_H^t$ we are interested in points $x\in\Sigma$ with the property
\beq
\phi_H^1(x)\in L_x\;.
\eeq
This notion was introduced and studied by Moser in \cite{Moser_A_fixed_point_theorem_in_symplectic_geometry}. Such points are called leaf-wise intersections.  We recall some notions from \cite{Albers_Frauenfelder_Leafwise_intersections_and_RFH}.
\begin{Def}\label{def:periodic_LI}
A leaf-wise intersection $x\in\Sigma$ is called periodic if the leaf $L_x$ is a closed orbit of the flow $\phi_f^t$.
\end{Def}

\begin{Def}\label{def:Moser_pair}
A pair $\Mp=(F,H)$ of Hamiltonian functions $F,H:S^1\times M\pf\R$ is called a Moser pair if it satisfies
\beq
F(t,\cdot)=0\quad\forall t\in[0,\tfrac12]\qquad\text{and}\qquad H(t,\cdot)=0\quad\forall t\in[\tfrac12,1]\;,
\eeq
and $F$ is of the form $F(t,x)=\rho(t)f(x)$ for some smooth map $\rho:S^1\to [0,\infty)$ with $\int_0^1\rho(t) dt=1$ and $f:M\to\R$. 
\end{Def}

\begin{Def}\label{def:set_of_half_constant_Hamiltonians}
We set
\beq
\mathcal{H}:=\{H\in C^\infty(S^1\times M)\mid H\text{ has compact support and } H(t,\cdot)=0\quad\forall t\in[\tfrac12,1]\}
\eeq
\end{Def}

\begin{Rmk}
It is easy to see that the $\Ham(M,\om)\equiv\{\phi_H^1\mid H\in\mathcal{H}\}$, e.g.~\cite{Albers_Frauenfelder_Leafwise_intersections_and_RFH}.
\end{Rmk}

Let $(M,\om=d\lambda)$ be an exact symplectic manifold. Then for a Moser pair $\Mp=(F,H)$ the perturbed Rabinowitz action functional is defined by
\bea
\A^\Mp:\L_M\times\R&\pf\R\\
(v,\eta)&\mapsto-\int_{S^1}v^*\lambda-\int_0^1H(t,v)dt-\eta\int_0^1F(t,v)dt
\eea
where $\L_M:=C^\infty(S^1,M)$. A critical point $(v,\eta)$ of $\A^\Mp$ is a solution of 
\beq\label{eqn:critical_points_eqn}
\left. 
\begin{aligned}
\partial_tv=\eta X_F(t,v)+X_H(t,v)\\
\int_0^1F(t,v)dt=0
\end{aligned}\right\}
\eeq
In \cite{Albers_Frauenfelder_Leafwise_intersections_and_RFH} it is proved that critical points of $\A^\Mp$ give rise to leaf-wise intersections. 

\begin{Prop}[\cite{Albers_Frauenfelder_Leafwise_intersections_and_RFH}]\label{prop:critical_points_give_LI}
Let $(v,\eta)$ be a critical point of $\A^\Mp$. Then $x:=v(0)\in f^{-1}(0)$ and
\beq
\phi_H^1(x)\in L_x
\eeq
thus, $x$ is a leaf-wise intersection.\\[1ex]
Moreover, the map $\Crit\A^\Mp\to\{\text{leaf-wise intersections}\}$ is injective unless there exists a periodic leaf-wise intersection (see Definition \ref{def:periodic_LI}).
\end{Prop}

\begin{Def}
A Moser pair $\Mp=(F,H)$ is of restricted contact type if the following four conditions hold.
\begin{enumerate}
\item $0$ is a regular value of $f$.
\item $df$ has compact support.
\item The hypersurface $f^{-1}(0)$ is a closed restricted contact type hypersurface of $(M,\lambda)$.
\item The Hamiltonian vector field $X_f$ restricts to the Reeb vector field on $f^{-1}(0)$.
\end{enumerate}
\end{Def}

\begin{Rmk}
If $\Sigma\subset (M,d\lambda)$ is a closed, bounding, restricted contact type hypersurface then it is easy to construct a restricted contact type Moser pair $\Mp_0=(F,0)$ with $\Sigma=f^{-1}(0)$. We fix such a Moser pair for the rest of the paper. Critical points $(v,\eta)$ of the unperturbed Rabinowitz action functional $\A^{(F,0)}$ are $\eta$-periodic Reeb orbits in $\Sigma$ or if $\eta=0$ constant loops in $\Sigma$. Their critical values are $\A^{(F,0)}(v,\eta)=-\eta$.
\end{Rmk}

In \cite{Cieliebak_Frauenfelder_Restrictions_to_displaceable_exact_contact_embeddings} Cieliebak and Frauenfelder construct a Floer homology for the unperturbed Rabinowitz action functional. This has been extended to the perturbed case by the first author and Frauenfelder in \cite{Albers_Frauenfelder_Leafwise_intersections_and_RFH}. The corresponding Floer equation for maps $u:\R\times S^1\to M$ and $\eta:\R\to\R$ is
\beq\label{eqn:gradient_flow_equation}\left.
\begin{aligned}
&\partial_su+J(s,t,u)\big(\partial_tu-\eta X_F(t,u)-X_{H_s}(t,u)\big)=0\\[1ex]
&\partial_s\eta-\int_0^1F(t,u)dt=0.
\end{aligned}
\;\;\right\}
\eeq
Here $J$ is a smooth $(s,t)$-dependent family of compatible almost complex structures and $H_s:S^1\times M\pf\R$ is a smooth $s$-dependent family of functions. Counting solutions of the $s$-independent equation modulo $\R$-action defines the boundary operator in Rabinowitz Floer homology. In this paper we do not need the full machinery of Rabinowitz Floer homology.

We set $w=(u,\eta)$ for solutions of \eqref{eqn:critical_points_eqn}. We will think of $w$ also as a map $w:\R\pf\L_M\times\R$.

\begin{Def}
The energy of a map $w=(u,\eta)$ is defined as
\beq
E(w):=\int_{-\infty}^\infty\int_0^1||\partial_s u||^2dtds+\int_{-\infty}^\infty|\partial_s \eta|^2ds\;.
\eeq
\end{Def}

The following has been established in \cite{Albers_Frauenfelder_Leafwise_intersections_and_RFH}.

\begin{Lemma}[\cite{Albers_Frauenfelder_Leafwise_intersections_and_RFH}, Lemma 2.7]\label{lemma:energy_estimate_for_gradient_lines}
Let $w$ solve \eqref{eqn:gradient_flow_equation} with limits 
\beq
w(\pm\infty):=\lim_{s\to\pm\infty}w(s)\in\Crit\A^{(F,H_\pm)}
\eeq
then we have 
\beq\label{eqn:energy estimate for gradient lines}
E(w)\leq\A^{(F,H_-)}(w(-\infty))-\A^{(F,H_+)}(w(+\infty))-\int_{-\infty}^\infty\int_0^1\min_M\partial_sH_s(t,\cdot)dtds
\eeq 
and equality holds if $\partial_sH_s=0$. 
\end{Lemma}

\begin{Thm}\label{thm:action_bounds_imply_compactness}
Let $w_n=(u_n,\eta_n)$ be a sequence of solutions of \eqref{eqn:gradient_flow_equation} for which there exists $a<b$ such that
\beq
a\leq\A^{(F,H)}\big(w_n(s)\big)\leq b\qquad\forall s\in\R\;.
\eeq
Then for every reparametrisation sequence $(\sigma_n)\subset\R$ the sequence $u_n(\cdot+\sigma_n)$ has a subsequence which converges in $C^\infty_\mathrm{loc}(\R\times S^1,M)$ and similarly for $\eta_n(\cdot+\sigma_n)$.
\end{Thm}

\section{Moduli spaces}\label{sec:moduli_space}

Let $k\geq1$ be a natural number. We choose a smooth family of functions $\beta_R\in C^\infty(\R,[0,1])$ satisfying
\begin{enumerate}
\item for $R\geq1$: $\beta_R'(s)\cdot s\leq0$ for all $s\in\R$, $\beta_R(s)=1$ for $s\in[0,(k+1)R]$, and $\beta_R(s)=0$ for $s\leq-1$ and $s\geq (k+1)R+1$, 
\item for $R\leq1$: $\beta_R(s)\leq R$ for all $s\in\R$ and $\mathrm{supp}\beta_R\subset[-1,k+2]$,
\item $\lim_{R\to\infty}\beta_R(s)=:\beta_\infty^+(s)$ and $\lim_{R\to\infty}\beta_R(s+(k+1)R)=:\beta_\infty^-(s)$ exists, where the limit is taken with respect to the $C^\infty_{loc}$ topology.
\end{enumerate} 

\begin{figure}[htb]
\psfrag{a}{$(k+1)R$}
\psfrag{b}{$(k+1)R+1$}
\psfrag{-1}{$-1$}
\psfrag{beta}{$\beta_R$}
\includegraphics[scale=1.1]{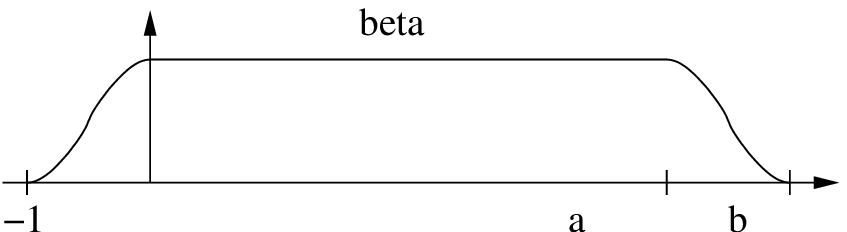}
\caption{The function $\beta_R$.}
\end{figure}

We fix a Hamiltonian function $H:[0,1]\times M\to\R$ with $H(t,\cdot)=0$ for all $t\in[\tfrac12,1]$ and set $K_R(s,t,x):=\beta_R(s)H(t,x)$ and
\beq
\Mp_R:=(F, K_R(s,t,x))\;.
\eeq
For every $R\geq0$ we define
\beq\label{eqn:moduli_space_M(R)}
\M(R):=\left\{w=(u,\eta)\in C^\infty(\R,\L_M\times\R)\bigg|\; 
\begin{aligned}
 &w\text{ solves }\eqref{eqn:gradient_flow_equation}\text{ with }\Mp_R\\
 &\lim_{s\to\pm\infty} u(s)\in\Sigma
\end{aligned}
\right\}\;.
\eeq
and 
\beq
\M[0,R]:=\left\{(r,w)\mid0\leq r\leq R\text{ and }w\in\M(r)\right\}\;.
\eeq
see figure \ref{fig:moduli_space_M(R)}.

\begin{Lemma}\label{lem:action_energy_bounds}
For $w\in\M(R)$  we have
\beq
E(w)\leq||H||\;.
\eeq
Moreover, for $R=0$ we have $E(w)=0$. Furthermore, 
\beq
-2\int_0^1 \max_M|H(t,\cdot)|dt\leq\A^{\Mp_R(s)}(w(s))\leq 2\int_0^1 \max_M|H(t,\cdot)|dt
\eeq
holds.
\end{Lemma}

\begin{proof}
We compute using $\p_sw+\nabla\A^{\Mp_R(s)}=0$
\bean
\underbrace{\A^{\Mp_0}(w(-\infty))}_{=0}-\A^{\Mp_R(s)}(w(s))&=-\int_{-\infty}^s\frac{d}{ds}\A^{\Mp_R(s)}(w(s))ds\\
&=-\int_{-\infty}^s d\A^{\Mp_R(s)}(w(s))\cdot\p_sw ds-\int_{-\infty}^s \frac{\p\A^{\Mp_R(s)}}{\p s}(w(s)) ds\\
&=-\int_{-\infty}^s\langle\underbrace{\nabla\A^{\Mp_R(s)}}_{=-\p_sw},\p_sw \rangle ds+\int_{-\infty}^s\int_0^1 \frac{\p K_R}{\p s}(s,t,u(s,t))dtds\\
&=\int_{-\infty}^s||\p_s w(s)||^2 ds+\int_{-\infty}^s \int_0^1 \frac{\p K_R}{\p s}(s,t,u(s,t))dtds\\
&\geq\int_{-\infty}^s\left(\frac{\p}{\p s} \beta_R(s)\right)\int_0^1H(t,u(s,t))dtds\\
&\geq\underbrace{\int_{-\infty}^0\left(\frac{\p}{\p s} \beta_R(s)\right)ds}_{=1}\cdot\int_0^1-\max_M|H(t,\cdot)|dt\\
&\;\;\;\;+\underbrace{\int_{0}^\infty\left(\frac{\p}{\p s} \beta_R(s)\right)ds}_{=-1}\cdot\int_0^1 \max_M|H(t,\cdot)|dt\\
&\geq-2\int_0^1 \max_M|H(t,\cdot)|dt
\eea
\begin{figure}[htb]
\psfrag{Sigma}{$\Sigma$}
\psfrag{R}{$(k+1)R$}
\psfrag{H}{$H$}
\includegraphics[scale=1]{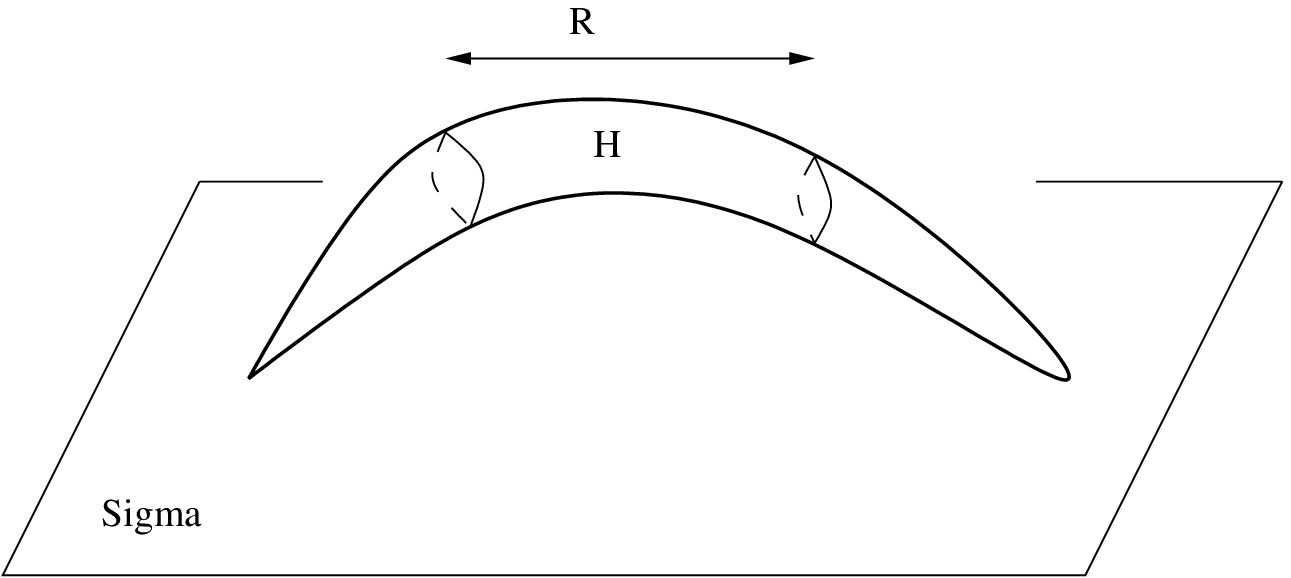}
\caption{An element of the moduli space $\M(R)$.}\label{fig:moduli_space_M(R)}
\end{figure}
The other inequality is proved by replacing $\int_{-\infty}^s$ by $\int_s^{\infty}$. The first two assertions follow from Lemma \ref{lemma:energy_estimate_for_gradient_lines} together with the observation that $\A^{(F,\beta_RH)}(w(\pm\infty))=0$ and 
\bea
-\int_{-\infty}^\infty\int_0^1\min_M\partial_sK_R(s,t,\cdot)dtds&=-\int_{-\infty}^\infty\int_0^1\min_M\{\p_s\beta_R(s)H(t,\cdot)\}dtds\\
&=-\int_{-\infty}^0\int_0^1\p_s\beta_R(s)\min_MH(t,\cdot)dtds\\
&\quad-\int_0^{+\infty}\int_0^1\p_s\beta_R(s)\max_MH(t,\cdot)dtds\\
&=\beta_R(0)||H||\leq||H||\;.
\eea
\end{proof}

\begin{Prop}\label{prop:compactness_moduli_spaces}
For any $R\geq0$ the moduli spaces $\M(R)$ and $\M[0,R]$ are compact. Moreover, $\M(0)\cong\Sigma$.
\end{Prop}

\begin{proof}
This follows easily as in the proof of Theorem A in \cite{Albers_Frauenfelder_Leafwise_intersections_and_RFH} as follows. Using Lemma \ref{lem:action_energy_bounds} we can apply Theorem \ref{thm:action_bounds_imply_compactness} to extract $C^\infty_{loc}$-convergent subsequences of any sequence $w_n(s-\sigma_n)$ where $w_n\in\M(R)$ and $(\sigma_n)\subset\R$. Then it is proved in \cite{Albers_Frauenfelder_Leafwise_intersections_and_RFH} that if the sequence does not converge to an element in $\M(R)$ there has to exist a non-constant gradient flow line $v$ of $\A^{(F,0)}$ with one asymptotic end on $\Sigma$. Therefore, there exists a Reeb orbit $x$ of period $\eta$ on $(\Sigma,\alpha)$ which is contractible in $M$. We conclude
\beq
E(v)\leq\limsup E(w_n)\leq||H||<\wp(\Sigma,\alpha)\;.
\eeq
On the other hand we can compute $E(v)$ by Lemma \ref{lemma:energy_estimate_for_gradient_lines}:
\beq
E(v)=|\eta|\geq\wp(\Sigma,\alpha)
\eeq
where the inequality follows from the definition of $\wp(\Sigma,\alpha)$. This contradiction shows that $\M(R)$ is compact. That $\M[0,R]$ is compact follows in the same way.

If $R=0$ then according to Lemma \ref{lem:action_energy_bounds} $E(w)=0$ for all $w\in\M(0)$. So $\p_s w(s)=0$ and $w(s)=(p,0)\in\Crit\A^{(F,0)}$ with $p\in\Sigma$ being the constant loop; thus, $\M(0)\cong\Sigma.$
\end{proof}
\begin{Rmk}
The moduli space $\M(R)$ is the zero-set of a Fredholm section $\F(R)$ of a Banach space bundle. The Fredholm index of $\F(R)$ equals $\ind\F(R)=\dim\Sigma$. The moduli space $\M(0)$ contains only constant gradient flow lines with Lagrange multiplier $\eta=0$, i.e.~$\M(0)\cong\Sigma$. As shown in \cite{Albers_Frauenfelder_Leafwise_intersections_and_RFH} the Morse-Bott property of the defining function $F$ implies that the Fredholm section is transverse to the zero-section for $R=0$. In particular, $\M(0)$ is a smooth manifold. Since $\M(R)$ is compact it is a smooth manifold for $R$ sufficiently small since transversality is an open property.
\end{Rmk}

\section{Cohomology operations}\label{sec:cohomology_operations}

We fix a natural number $k\geq1$. The moduli space $\M(R)$ carries an evaluation map
\bea
\ev_R:\M(R)&\pf M^k=M\times\ldots\times M\\
w=(u,\eta)&\mapsto \big(u(R,0),u(2R,0),\ldots,u(kR,0)\big)\;.
\eea
For $R=0$ the evaluation map is under the identification $\M(0)\cong\Sigma$ the diagonal embedding $\Sigma\hpf\Delta_\Sigma\subset M^k$. Similarly, $\M[0,R]$ has an evaluation map $\mathrm{EV}(r,w):=\ev_r(w)$ to $M^k$.

Next, we define a Morse theoretic realization of the cohomology operation 
\bea\label{eqn:cohomology_op_Theta}
\Theta:\H^*(M)\otimes\ldots\otimes\H^*(M)\otimes\H_*(\Sigma)&\pf\H_*(\Sigma)\\
a_1\otimes \ldots\otimes a_k\otimes b&\mapsto \iota^*(a_1\cup\ldots\cup a_k)\cap b
\eea
where $\iota:\Sigma\hpf M$ is the inclusion map. 

\begin{Rmk}
Since the symplectic manifold $(M,\om)$ is exact it either has boundary or is non-compact. In the following we will choose Morse functions on $M$ in order to model singular (co-)homology by Morse (co-)homology. For this we need to restrict to a certain class of Morse functions. We are always considering the negative gradient flow. If $M$ has boundary then we assume that the (positive) gradient of the Morse function points outward along the boundary of $M$. If $M$ is non-compact we assume that the Morse function is proper and bounded from below. Under these assumptions standard Morse (co-)homology can be defined and is isomorphic to singular (co-)homology. From now on we assume whenever we choose Morse functions on $M$ they are in the just prescribed class.
\end{Rmk}

For that we choose Morse functions $f_1,\ldots,f_k:M\to\R$ and $f_*:\Sigma\to\R$ and Riemannian metrics $g_1,\ldots,g_k,g_*$. We set for critical points $x_j\in\Crit(f_j)$ and $x_*^\pm\in\Crit(f_*)$
\beqn
\M(0,x_1,\ldots,x_k,x_*^-,x_*^+):=\left\{w=(u,\eta)\in\M(0)\bigg|
\begin{aligned} 
&u(-\infty)\in W^u(x_*^-,f_*),\,u(+\infty)\in W^s(x_*^+,f_*)\\
&\ev_0(u)\in W^s(x_1,f_1)\times\ldots\times W^s(x_k,f_k)
\end{aligned}
\right\}
\eeq
see figure \ref{fig:moduli_space_M(0,x,x)}. In particular, $\M(0,x_1,\ldots,x_k,x_*^\pm)=\emptyset$ unless $\bigcap W^s(x_i,f_i)\cap\Sigma\neq\emptyset$. We denote by $\CM^*(f)$ the Morse chain complex associated to a Morse-Smale pair $(f,g)$.

\begin{figure}[htb]
\psfrag{Sigma}{$\Sigma$}
\psfrag{R}{$R$}
\psfrag{x1}{$x_1$}
\psfrag{xn}{$x_k$}
\psfrag{x+}{$x_*^+$}
\psfrag{x-}{$x_*^-$}
\psfrag{d}{$\cdots$}
\includegraphics[scale=1]{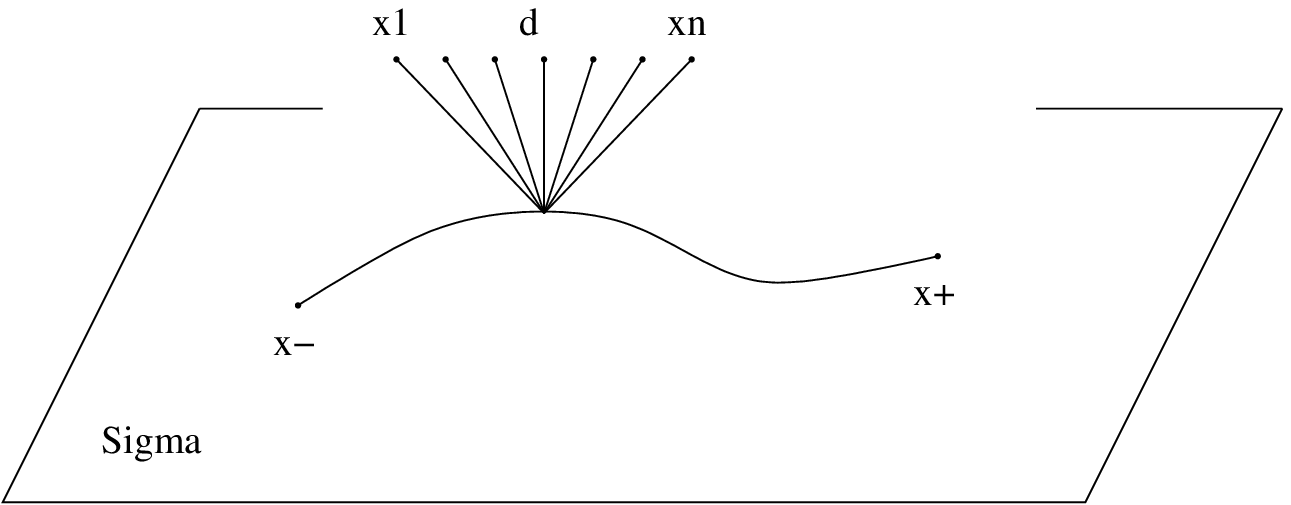}
\caption{An element of the moduli space $\M(0,x_1,\ldots,x_k,x_*^-,x_*^+)$.}\label{fig:moduli_space_M(0,x,x)}
\end{figure}

\begin{Prop}\label{prop:cohomology_op_1}
For generic Morse functions $f_i$ and Riemannian metrics $g_i$ the moduli space $\M(0,x_1,\ldots,x_k,x_*^-,x_*^+)$ is a smooth manifold.  Moreover, the map defined by
\bea
\theta_0:\CM^*(f_1)\otimes\ldots\otimes\CM^*(f_k)\otimes\CM_*(f_*)&\pf\CM_*(f_*)\\
x_1\otimes \ldots\otimes x_k\otimes x_*^-&\mapsto \sum_{x_*^+}\#_2\M(0,x_1,\ldots,x_k,x_*^-,x_*^+)\cdot x_*^+
\eea
defines a chain map which on homology agrees with the cohomology operation $\Theta$.
\end{Prop}

\begin{Rmk}
$\#_2\M$ denotes the parity of the set $\M$ if it is finite and zero otherwise.
\end{Rmk}

\begin{proof}
It follows from standard Morse theory that for generic Morse functions $f_i$ and Riemannian metrics $g_i$ we have
\beq
W^s(x_1,f_1)\times\ldots\times W^s(x_k,f_k)\pitchfork \Delta_\Sigma\subset M^k
\eeq 
where $\Delta_\Sigma\subset M^k$ is the diagonal embedding of $\Sigma$. Since for $R=0$ the evaluation map $\ev_0:\M(0)\to M^k$ is this diagonal embedding we conclude
\beq
W^s(x_1,f_1)\times\ldots\times W^s(x_k,f_k)\pitchfork \ev_0\;.
\eeq
Finally, choosing $f_*$ and $g_*$ generic we see that the moduli space $\M(0,x_1,\ldots,x_k,x_*^-,x_*^+)$ is smooth. That $\theta_0=\Theta$ after identifying Morse homology with singular homology is again standard Morse theory, see \cite{Schwarz_Morse_homology}.
\end{proof}

Since the moduli space $\M(R)$ is the zero-set of a Fredholm section $\F(R)$ and since $\M(R)$ is compact we can choose an arbitrarily small abstract perturbation of the Fredholm section $\F(R)$ such that the zero-set $\Mt(R)$ of the perturbed Fredholm section $\widetilde{\F}(R)$ is a smooth compact finite-dimensional manifold. Since $\M(R)$ is already transverse for sufficiently small $R$ we can arrange that $\Mt(R)=\M(R)$ for small $R$. We point out that $\M[0,R]$ has a natural projection to $[0,R]$. The same abstract perturbation procedure gives rise to smooth perturbed moduli spaces $\Mt[0,R]$ where the perturbation can be chosen with fixed ends, i.e., the fibers over $0$ resp.~$R$ are  $\M(0)$ resp.~ $\Mt(R)$. The perturbed moduli space $\Mt(R)$ resp.~$\Mt[0,R]$ still carries an evaluation map $\evt(R):\Mt(R)\to M^k$ resp.~$\EVt:\Mt[0,R]\to M^k$. 

For Morse functions $f_1,\ldots,f_k:M\to\R$ and $f_*:\Sigma\to\R$, Riemannian metrics $g_1,\ldots,g_k,g_*$, and  critical points $x_j\in\Crit(f_j)$ and $x_*^\pm\in\Crit(f_*)$ we set
\beqn
\Mt(R,x_1,\ldots,x_k,x_*^-,x_*^+):=\left\{w=(u,\eta)\in\Mt(R)\bigg|
\begin{aligned} 
&u(-\infty)\in W^u(x_*^-,f_*),\,u(+\infty)\in W^s(x_*^+,f_*)\\
&\evt_R(u)\in W^s(x_1,f_1)\times\ldots\times W^s(x_k,f_k)
\end{aligned}
\right\}
\eeq
see figure \ref{fig:moduli_space_M(R,xxx)}.

\begin{figure}[htb]
\psfrag{Sigma}{$\Sigma$}
\psfrag{R}{$(k+1)R$}
\psfrag{x1}{$x_1$}
\psfrag{xn}{$x_k$}
\psfrag{x+}{$x_*^+$}
\psfrag{x-}{$x_*^-$}
\psfrag{d}{$\cdots$}
\includegraphics[scale=1]{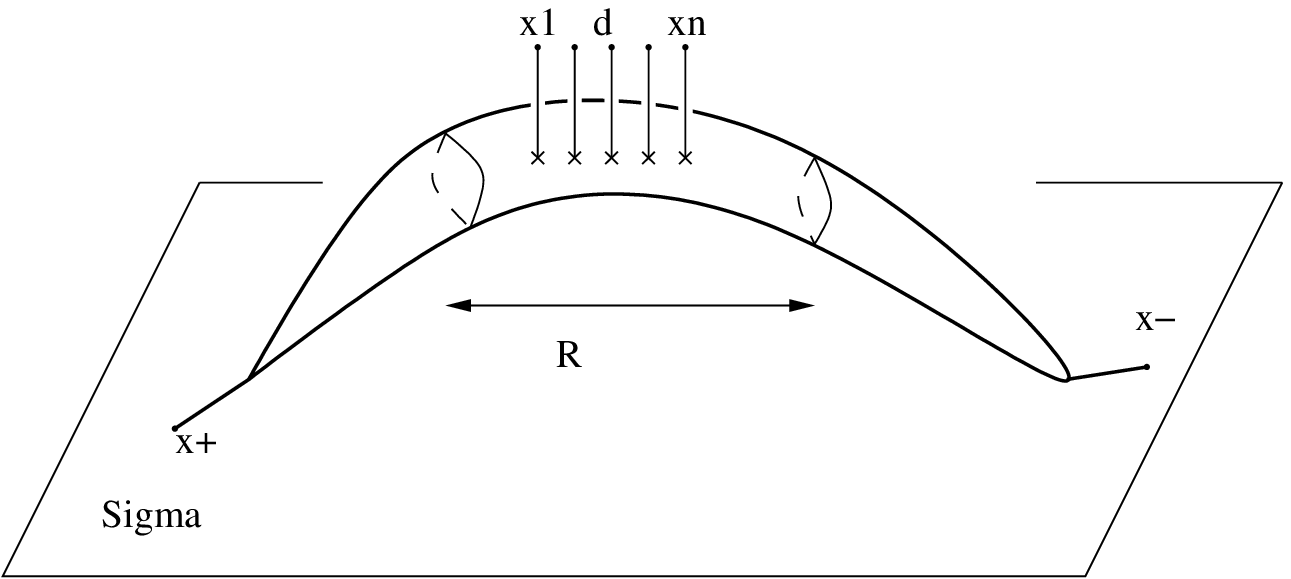}
\caption{An element of the moduli space $\Mt(R,x_1,\ldots,x_k,x_*^-,x_*^+)$.}\label{fig:moduli_space_M(R,xxx)}
\end{figure}

\begin{Prop}\label{prop:cohomology_op_2}
For generic Morse functions $f_i$ and Riemannian metrics $g_i$ the moduli space $\Mt(R,x_1,\ldots,x_k,x_*^-,x_*^+)$ is a smooth manifold.  Moreover, the map defined by
\bea
\theta_R:\CM^*(f_1)\otimes\ldots\otimes\CM^*(f_k)\otimes\CM_*(f_*)&\pf\CM_*(f_*)\\
x_1\otimes \ldots\otimes x_k\otimes x_*^-&\mapsto \sum_{x_*^+}\#_2\Mt(R,x_1,\ldots,x_k,x_*^-,x_*^+)\cdot x_*^+
\eea
defines a chain map which is chain homotopic to the cohomology operation $\theta_0$.
\end{Prop}

\begin{proof}
For generic Morse functions $f_i$, generic Riemannian metrics $g_i$, and generic perturbation of the Fredholm section we have
\beq
W^s(x_1,f_1)\times\ldots\times W^s(x_k,f_k)\pitchfork \evt_R\;.
\eeq 
For generic $f_*$ and $g_*$ the stable and unstable manifolds are transversal to the evaluation maps $u\mapsto u(\pm\infty)$ and thus the moduli space $\Mt(R,x_1,\ldots,x_k,x_*^-,x_*^+)$ is smooth. The map $\theta_R$ is a chain map by standard Morse theory since $\Mt(R)$ is compact.

To prove that $\theta_R$ is chain homotopic to $\theta_0$ we recall that $\Mt[0,R]$ is perturbed while keeping the ends $\M(0)$ and $\Mt(R)$ fixed. Perturbing further (with fixed ends), the evaluation map $\EVt:\Mt[0,R]\to M^k$ will be transverse to all products of unstable manifolds. Thus, the moduli space $\Mt[0,R]$ together with $\EVt$ induces a cobordism between the moduli spaces  $\Mt(R,x_1,\ldots,x_k,x_*^\pm)$ and $\Mt(0,x_1,\ldots,x_k,x_*^\pm)$. This give rise to a chain homotopy operator between $\theta_0$ and $\theta_R$.
\end{proof}

\section{Proof of Theorem \ref{thm:main}}

We assume that $\psi$ has only finitely many leaf-wise intersections since otherwise we are done. We set $k:=\mathrm{cl}(\Sigma,M)$. It has been proved in \cite{Albers_Frauenfelder_Leafwise_intersections_and_RFH} that $\nu_{\text{leaf}}(\psi)\geq1$ if $||\psi||<\wp(\Sigma,\alpha)$. Thus, we may assume that $k\geq1$. We choose a Hamiltonian $H:[0,1]\times M\to\R$ with $H(t,\cdot)=0$ for $t\in[\tfrac12,1]$ such that $\phi_H^1=\psi$.  Let $\LI\subset\Sigma$ be the set of leaf-wise intersections of $\psi$.

We choose Morse functions $f_1,\ldots,f_k$ and Riemannian metrics $g_1,\ldots,g_k$ on $M$ and $f_*$, $g_*$ on $\Sigma$ with the following properties
\begin{enumerate}
 \item For $x_i\in\Crit(f_i)$ with Morse index $\Morse(x_i,f_i)\neq0$ we have $W^s(x_i,f_i)\cap\LI=\emptyset$.
 \item For all $n\in\N$ the evaluation maps $\evt_n:\Mt(n)\to M^k$ are transverse to the products of stable manifolds of the $(f_i,g_i)$.
 \item For all $n\in\N$ the evaluation maps at $\pm\infty$ are transverse to all stable and unstable manifolds of $(f_*,g_*)$.
\end{enumerate}
The second and third property holds for a fixed $n$ for a Baire set of Morse functions and Riemannian metrics as explained in the proofs of Propositions \ref{prop:cohomology_op_1} and \ref{prop:cohomology_op_2}. Intersecting these Baire sets over $n\in\N$ we see (2) and (3) is a generic property.  Avoiding finitely many unstable manifolds not of top dimension is clearly also a generic condition.

By definition of $k$ we find cohomology classes $a_1,\ldots,a_k\in\H^{\geq1}(M)$ with $\iota^*(a_1\cup\ldots\cup a_k)\neq0\in\H^*(\Sigma)$. Thus, the cohomology operation $\Theta$, see equation \eqref{eqn:cohomology_op_Theta}, is non-zero. According to Propositions \ref{prop:cohomology_op_1} and \ref{prop:cohomology_op_2} the cohomology operations $\theta_R$ then have to be non-zero for all $R\geq0$. Thus, we can find critical points $x_i\in\Crit(f_i)$ with Morse index $\Morse(x_i,f_i)\neq 0$ and $x_*^\pm\in\Crit(f_*)$ such that for all $n\in\N$ the moduli spaces 
\beq
\Mt(n,x_1,\ldots,x_k,x_*^-,x_*^+)\neq\emptyset\;.
\eeq
In fact, the choice of $x_i\in\Crit(f_i)$ will depend on $n$ in general. For notational convenience we suppress the $n$. In fact, there exists a subsequence $n_k$ for which we can choose $x_i$ fixed. We conclude
\beq
\Mt(n)\neq\emptyset\;.
\eeq
If the unperturbed moduli space $\M(n)$, see equation \eqref{eqn:moduli_space_M(R)}, were empty then for sufficiently small perturbations of the Fredholm section also the moduli space $\Mt(n)$ would be empty. Indeed, if $\M(n)=\emptyset$ then the corresponding Fredholm section is transverse. Since $\M(n)$ is compact a sufficiently small perturbation of the Fredholm section remains empty and thus transverse. Therefore, we conclude that
\beq
\M(n)\neq\emptyset\;.
\eeq
Now we can choose a sequence $w_n\in\M(n)$, that is, $w_n$ solves equation \eqref{eqn:gradient_flow_equation} with Hamiltonian perturbation $\beta_n(s)H(t,x)$. Since $\Mt(n)\cap_{\evt_n}\big(W^s(x_1,f_1)\times\ldots\times W^s(x_k,f_k)\big)\neq\emptyset$ and all evaluation maps are transverse a similar argument as above allows us to conclude that 
\beq
\ev_n(w_n)\in{W^s(x_1,f_1)}\times\ldots\times{W^s(x_k,f_k)}\;.
\eeq
By Lemma \ref{lem:action_energy_bounds} the action $\A^{\Mp_n}(w_n(s))$ is uniformly bounded and we can apply Theorem \ref{thm:action_bounds_imply_compactness}. We consider the following sequences
\beq
w_n(s+jn),\quad j=0,\ldots, k+1\;.
\eeq
For $j=1,\ldots,k$ these sequences converge (after choosing a subsequence) to solutions $\varpi^j=(u^j,\eta^j):\R\to\L_M\times\R$ of
\beq
\left.
\begin{aligned}
&\partial_su^j+J(s,t,u^j)\big(\partial_tu^j-\eta^j X_F(t,u)-X_{H}(t,u^j)\big)=0\\[1ex]
&\partial_s\eta^j-\int_0^1F(t,u^j)dt=0
\end{aligned}
\;\;\right\}
\eeq
of energy
\beq
E(\varpi^j)\leq\limsup E(w_n)\leq||H||\;.
\eeq
The sequence $w_n(s)$ converges to a solution $\varpi^0=(u^0,\eta^0)$ of
\beq
\left.
\begin{aligned}
&\partial_su^0+J(s,t,u^0)\big(\partial_tu^0-\eta^0 X_F(t,u)-\beta^+_\infty(s)X_{H}(t,u^0)\big)=0\\[1ex]
&\partial_s\eta^0-\int_0^1F(t,u^0)dt=0
\end{aligned}
\;\;\right\}
\eeq
and the sequence $w_n(s+(k+1)n)$ converges to a solution $\varpi^{k+1}=(u^{k+1},\eta^{k+1})$ of
\beq
\left.
\begin{aligned}
&\partial_su^{k+1}+J(s,t,u^{k+1})\big(\partial_tu^{k+1}-\eta^{k+1} X_F(t,u)-\beta^-_\infty(s)X_{H}(t,u^{k+1})\big)=0\\[1ex]
&\partial_s\eta^{k+1}-\int_0^1F(t,u^{k+1})dt=0
\end{aligned}
\;\;\right\}
\eeq
where $\beta^\pm_\infty$ are defined at the beginning of section \ref{sec:moduli_space}. In particular, 
\beq
(y^j_\pm,\eta^j_\pm):=\varpi^j(\pm\infty)\in\Crit\A^{(F,H)},\quad j=1,\ldots,k
\eeq 
and
\beq
(y^0_+,\eta^0_+):=\varpi^0(+\infty),\;(y^{k+1}_-,\eta^{k+1}_-):=\varpi^{k+1}(-\infty)\in\Crit\A^{(F,H)}
\eeq
are critical points. Thus, $y^j_\pm(0)$ are leaf-wise intersections, see Proposition \ref{prop:critical_points_give_LI}. Moreover, they are ordered by action as follows
\bea
\A^{(F,H)}(y^0_+,\eta^0_+)\geq\A^{(F,H)}(y^1_-,\eta^1_-)&\geq\A^{(F,H)}(y^1_+,\eta^1_+)\geq\A^{(F,H)}(y^2_-,\eta^2_-)\geq\ldots\ldots\\
&\ldots\ldots\geq\A^{(F,H)}(y^k_+,\eta^k_+)\geq\A^{(F,H)}(y^{k+1}_-,\eta^{k+1}_-)\;.
\eea
This follows directly from the choice of the sequences $w_n(s+jn)$. If $\A^{(F,H)}(\varpi^j_-)=\A^{(F,H)}(\varpi^j_+)$ then by Lemma \ref{lemma:energy_estimate_for_gradient_lines} $E(\varpi^j)=\A^{(F,H)}(\varpi^j_-)-\A^{(F,H)}(\varpi^j_+)=0$ and thus $\p_s \varpi^j(s)=0$. We recall that $\ev_n(w_n)\in{W^s(x_1,f_1)}\times\ldots\times{W^u(x_k,f_k)}$. By $C^\infty_{loc}$-convergence of $w_n(s+jn)\to\varpi^j(s)=(u^j,\eta^j)$ we conclude
\beq
u^j(0,0)\in\overline{W^s(x_j,f_j)}
\eeq
where $\varpi^j=(u^j,\eta^j)$. Now, since $\p_s\varpi^j(s)=0$, i.e.~$\varpi^j(s)=(y^j_-,\eta^j_-)=(y^j_+,\eta^j_+)$, we have
\beq
u^j(0,0)=y^j_-(0)=y^j_+(0)\in\overline{W^s(x_j,f_j)}\;.
\eeq
But this contradicts assumption (1) that $W^s(x_j,f_j)\cap\LI=\emptyset$ since $\p\overline{W^s(x_j,f_j)} $ is composed out of stable manifolds of critical points of higher indices. 
Therefore, $\p_s\varpi^j\neq0$ for all $j=1,\ldots,k$ and we conclude
\beq
\A^{(F,H)}(y^j_-,\eta^j_-)>\A^{(F,H)}(y^j_+,\eta^j_+),\quad j=1,\ldots,k\;.
\eeq
If we denote $(y_l,\eta_l):=(y^l_-,\eta^l_-)$ for $l=1,\ldots,k$ and $(y_{k+1},\eta_{k+1}):=(y^k_+,\eta^k_+)$ we have
\beq\label{eqn:inequ_chain_for_action_values}
\A^{(F,H)}(y^0_+,\eta^0_+)\geq\A^{(F,H)}(y_1,\eta_1)>\ldots>\A^{(F,H)}(y_{k+1},\eta_{k+1})\geq\A^{(F,H)}(y^{k+1}_-,\eta^{k+1}_-)\;.
\eeq
In particular, all $(y_l,\eta_l)$ are different as critical points of $\A^{(F,H)}$. To finally conclude that the leaf-wise intersections $y_1(0),\ldots,y_{k+1}(0)$ are all distinct we recall from Proposition \ref{prop:critical_points_give_LI} that two critical points of $\A^{(F,H)}$ can give rise to the same leaf-wise intersection only if some $y_j(0)$ is a periodic leaf-wise intersection. If $y_l(0)=y_{l'}(0)$ then as in the proof of \cite[Lemma 2.19]{Albers_Frauenfelder_Leafwise_intersections_and_RFH} we conclude that
\beq\label{eqn:ineq_periodic_leafwise_intersection}
|\A^{(F,H)}(y_l,\eta_l)-\A^{(F,H)}(y_{l'},\eta_{l'})|\geq\wp(\Sigma,\alpha)>||H||\;.
\eeq
On the other hand we estimate for $(y^0_+,\eta^0_+)=\varpi^0(+\infty)$
\bean
\A^{(F,H)}(\varpi^0(+\infty))-\A^{(F,0)}(\varpi^0(-\infty))&=\int_{-\infty}^\infty\frac{d}{ds}\A^{(F,\beta^+_\infty(s) H)}(\varpi^0(s))ds\\
&=\int_{-\infty}^\infty d\A^{(F,H)}(\varpi^0)\cdot\p_s\varpi^0\; ds+\int_{-\infty}^\infty \frac{\p\A^{(F,\beta^+_\infty(s) H)}}{\p s}(\varpi^0)ds\\
&=-\int_{-\infty}^\infty||\p_s \varpi^0||^2 ds-\int_{-\infty}^\infty\int_0^1 \frac{\p\big( \beta^+_\infty(s) H\big)}{\p s}(t,u^0)dtds\\
&\leq-\int_{-\infty}^\infty\Big(\underbrace{\frac{\p}{\p s} \beta^+_\infty(s)}_{\geq0}\Big)\int_0^1H(t,u^0(s,t))dtds\\
&\leq -\underbrace{\int_{-\infty}^\infty\left(\frac{\p}{\p s}\beta^+_\infty(s)\right)ds}_{=1}\cdot\int_0^1\min_MH(t,\cdot)dt\\
&=-\int_0^1\min_MH(t,\cdot)dt\;.
\eea
The critical point $\varpi^0(-\infty)$ of $\A^{(F,0)}$ corresponds either to a Reeb orbit or constant loop in $\Sigma$. Since $E(\varpi^0)\leq||H||$ the same argument as in Proposition \ref{prop:compactness_moduli_spaces} rules out the former case. In particular, $\A^{(F,0)}(\varpi^0(-\infty))=0$ and therefore
\beq
\A^{(F,H)}(y^0_+,\eta^0_+)\leq-\int_0^1\min_MH(t,\cdot)dt\;.
\eeq 
Analogously, for $(y^{k+1}_-,\eta^{k+1}_-)=\varpi^{k+1}(-\infty)$
\beq
\A^{(F,H)}(y^{k+1}_-,\eta^{k+1}_-)\geq-\int_0^1\max_MH(t,\cdot)dt
\eeq 
is derived. Thus, if we assume  $y_l(0)=y_{l'}(0)$ we have the following inequalities
\bea
||H||&=-\int_0^1\min_MH(t,\cdot)dt+\int_0^1\max_MH(t,\cdot)dt\\
&\geq \A^{(F,H)}(y^0_+,\eta^0_+)-\A^{(F,H)}(y^{k+1}_-,\eta^{k+1}_-)\\
&\geq |\A^{(F,H)}(y_l,\eta^l)-\A^{(F,H)}(y_{l'},\eta^{l'})|\\
&\geq\wp(\Sigma,\alpha)\\
&>||H||
\eea
where the second inequality follows from \eqref{eqn:inequ_chain_for_action_values} and the last two from \eqref{eqn:ineq_periodic_leafwise_intersection}. This contradiction shows that the the leaf-wise intersection points $y_1(0),\ldots,y_{k+1}(0)$ are all distinct. This finishes the proof of Theorem \ref{thm:main}.

\subsubsection*{Acknowledgments}
This article was written during visits of the authors at the Institute for Advanced Study, Princeton. The authors thank the Institute for Advanced Study for their stimulating working atmospheres. The authors are grateful to Alberto Abbondandolo, Urs Frauenfelder, and Helmut Hofer for helpful discussions. 

This material is based upon work supported by the National Science Foundation under agreement No.~DMS-0635607 and DMS-0903856. Any opinions, findings and conclusions or recommendations expressed in this material are those of the authors and do not necessarily reflect the views of the National Science Foundation.

\bibliographystyle{amsalpha}
\bibliography{../../../Bibtex/bibtex_paper_list}
\end{document}